\documentclass[11pt, reqno]{amsart}
\usepackage{amssymb}
\usepackage{amsmath}
\usepackage{enumerate}
\usepackage{mathrsfs}
\usepackage{amsfonts}  
\usepackage{color}
\usepackage{vmargin}
\usepackage{amsthm}
\usepackage{graphicx}  
\usepackage{esvect}
\usepackage{MnSymbol}

\def\XXint#1#2#3{{\setbox0=\hbox{$#1{#2#3}{\int}$ }
\vcenter{\hbox{$#2#3$ }}\kern-.6\wd0}}

\long\def\symbolfootnote[#1]#2{\begingroup%
\def\thefootnote{\fnsymbol{footnote}}\footnote[#1]{#2}\endgroup}

\setmarginsrb{20mm}{20mm}{20mm}{20mm}{10mm}{10mm}{10mm}{10mm}

\newtheoremstyle{remark}
  {}{}{}{}{\bfseries}{.}{.5em}{{\thmname{#1 }}{\thmnumber{#2}}{\thmnote{ (#3)}}}
\theoremstyle{remboldstyle}

\RequirePackage{amsthm}

\newtheorem{tw}{Theorem}[section]
\newtheorem{defi}[tw]{Definition}

\newtheorem{lem}[tw]{Lemma}
\newtheorem{cor}[tw]{Corollary}
\newtheorem{prop}[tw]{Proposition}
\newtheorem{remark}[tw]{Remark}
\newtheorem{ex}[tw]{Example}

\def\={\hspace{-3mm}&=&\hspace{-3mm}}

\usepackage{tikz}
\usetikzlibrary{arrows,calc,patterns,cd}
\usepackage{wrapfig}

\let\oldmarginpar\marginpar
\renewcommand{\marginpar}[2][rectangle,draw,text width= 1.5cm,rounded corners]{
    \oldmarginpar{
    \scriptsize \tikz \node at (0,0) [#1]{#2};}
    }
\makeatletter
\@namedef{subjclassname@2020}{\textup{2020} Mathematics Subject Classification}
\makeatother

\begin{document}

\date{}

\title[Borel regularity, Lusin's theorem, and the existence of Borel representatives]{\bf Borel regularity is equivalent to Lusin's theorem and the existence of Borel representatives}


\author[Alvarado]{Ryan Alvarado}
\address{Amherst College, Amherst, Massachusetts}
\email{rjalvarado\@@amherst.edu}

\author[G\'orka]{Przemys{\l}aw G\'orka}
\address{Warsaw University of Technology, Warsaw, Poland}
\email{przemyslaw.gorka\@@pw.edu.pl}

\author[S{\l}abuszewski]{Artur S{\l}abuszewski}
\address{Warsaw University of Technology, Warsaw, Poland}
\email{artur.slabuszewski\@@pw.edu.pl}





\maketitle
\begin{abstract}
In this article, we characterize both Lusin's theorem and the existence of Borel representatives via the regularity properties of the measure in general topological measure spaces. As a corollary, we prove that Borel regularity of the measure is both a necessary and sufficient condition for these results to hold true in metric measure spaces.
\end{abstract}
\bigskip

%

%
%
\section{Introduction}
Recall that the classical (weak) Lusin theorem (see \cite{L12}) states that every (Lebesgue) measurable function $u:[0,1]\to\mathbb{R}$ is continuous on ``almost all" of its domain, in the sense that, for every $\varepsilon\in(0,\infty)$, there exists a closed set $C_\varepsilon\subseteq[0,1]$ such that $|[0,1]\setminus C_\varepsilon|<\varepsilon$ (where $|\cdot|$ denotes the 1-dimensional Lebesgue measure) and the restriction of $u$ to $C_\varepsilon$ is continuous; see \cite{C27} for  an alternative proof of this result. In fact, by the Tietze extension theorem, the following strong version of Lusin's theorem is valid: for every  measurable function $u:[0,1]\to\mathbb{R}$ and for every $\varepsilon\in(0,\infty)$, there exists a closed set $C_\varepsilon\subseteq[0,1]$ and a continuous function $g:[0,1]\to\mathbb{R}$ such that $|[0,1]\setminus C_\varepsilon|<\varepsilon$ and $g=u$ pointwise in $C_\varepsilon$. As a consequence of Lusin's theorem, every measurable function $u:[0,1]\to\mathbb{R}$ has a Borel representative, that is, there exists a Borel function  $f:[0,1]\to\mathbb{R}$ (called a \textit{Borel representative of $u$}) such that $f=u$ pointwise outside of a null set.
It is well known that these classical results extend to functions defined on more general topological measure spaces, provided that the measure and topological space exhibit a sufficient degree of regularity; see, for example, \cite{Bog,Federer,Feldman,Niizeki,Stone,Wage,Zakon}.  
%
%
The main goal of this article is to revisit the weak and strong Lusin theorems and the existence of Borel representatives in the setting of general topological spaces equipped with a Borel measure (namely, a nonnegative measure  that is defined on $\sigma$-algebra which simply contains all Borel sets) and identify conditions that are both necessary and sufficient for these results to hold true.

%

Turning first to the optimal conditions for Lusin's theorems, in Theorem~\ref{wLusin_equiv}, we prove that the weak Lusin theorem, formulated on a general topological space $X$ equipped with a Borel measure $\mu$, is equivalent to the measure $\mu$ being strongly regular, meaning that for each measurable set $E\subseteq X$ and for each $\varepsilon\in(0,\infty)$, there exist a closed set $C$ and an open set $O$ satisfying $C\subseteq E\subseteq O$ and $\mu(O\setminus C)<\varepsilon$. As we illustrate in Section~\ref{sect:weaklusin}, a wide variety of Borel measures satisfy this  regularity condition. In particular, every Borel regular measure (that is, every Borel measure with the property that each measurable set is contained in a Borel set of equal measure) on a metric space which is finite on balls is strongly regular. In Theorem~\ref{str_Lus_equiv}, we go further to prove that this strong regularity property, together with a weak Tietze extension-type condition, fully characterizes the strong Lusin theorem in the same general setting.

%
Next, while it is possible to obtain the existence of Borel representatives from Lusin's theorem (or, equivalently, from strong regularity of the measure) in more general topological measure spaces, it turns out that such assumptions are unnecessarily strong, in general.
Indeed, for $\sigma$-finite, Borel regular measures on a topological space, every measurable function has a Borel representative; see, for example, \cite[Proposition~3.3.23]{HKST}.\footnote{Every strongly regular measure is Borel regular, but not vice versa; see Corollary~\ref{str_reg_impl_Fsigma} and Remark~\ref{r52}. See also Examples~\ref{ex:wkL-1} and \ref{weakl}, which provide instances of $\sigma$-finite and finite, Borel regular measures that are not strongly regular (and hence, the weak Lusin theorem fails to hold in these settings).} We show in Example~\ref{ex:Brep} that the $\sigma$-finite assumption cannot be omitted, in general, for this result to hold true. Bringing this topic to a natural conclusion, in Theorem~\ref{rep-reg}, we prove that Borel regularity is also necessary for the existence of Borel representatives, even for non-$\sigma$-finite measures.

As a corollary of  our characterizations in Theorems~\ref{wLusin_equiv}, \ref{str_Lus_equiv}, and \ref{rep-reg}, we prove that under certain mild conditions on the topological measure space, Borel regularity fully characterizes both the weak and strong versions of Lusin's theorem and the existence of Borel representatives; see Theorem~\ref{main}. 
It is important to note that the conditions on the topological measure space in Theorem~\ref{main} occur naturally. In fact, by specializing Theorem~\ref{main} to the setting of metric spaces equipped with a Borel measure, we obtain the following characterization. Here, and thereafter, we let $\overline{\mathbb{R}}:=[-\infty,\infty]$.

\begin{tw}\label{main-intro}
Let $(X,d)$ be a metric space and suppose that $\mu$ is a nonnegative measure that is defined on a $\sigma$-algebra which contains all Borel sets in $X$ and has the property that every ball has finite measure. Then, the following statements are equivalent. 
\begin{enumerate}[(a)]
\item The measure $\mu$ is Borel regular, i.e., for every measurable set $E\subseteq X$, there exists a Borel set $B\subseteq X$ such that $E\subseteq B$ and $\mu(B)=\mu(E)$.

\item Each measurable function has a Borel representative, i.e., for every measurable function $u:X\to\overline{\mathbb{R}}$, there exists a Borel set $N\subseteq X$ and a Borel function $f:X\to\overline{\mathbb{R}}$ such that 
$\mu(N)=0$ and $f=u$ pointwise in $X\setminus N$.

\item The weak Lusin theorem holds, i.e., for every measurable function $u:X\to\overline{\mathbb{R}}$ and any $\varepsilon\in(0,\infty)$, there exists a closed set $C_{\varepsilon}\subseteq X$ such that 
$\mu(X\setminus C_{\varepsilon}) < \varepsilon$ and the restriction of $u$ to $C_{\varepsilon}$ is continuous.

\item The strong  Lusin theorem holds, i.e., for every measurable function $u:X\to\overline{\mathbb{R}}$ and any $\varepsilon\in(0,\infty)$, there exists a closed set $C_{\varepsilon}\subseteq X$ and a continuous function $g:X\to\overline{\mathbb{R}}$ such that 
$\mu(X\setminus C_{\varepsilon}) < \varepsilon$ and $g=u$ pointwise in $C_{\varepsilon}$.
\end{enumerate}
\end{tw}

Although we formulate our main results for functions taking values in $\overline{\mathbb{R}}$, these results could also be equally well formulated for functions taking values in certain general spaces, but we omit such generalizations.

The remainder of this article is organized as follows. In Section~\ref{sect:Borelreg}, we review several regularity properties of measures, including Borel regularity, and establish an equivalence between them in Theorem~\ref{in-out-reg}. 
In Section~\ref{sect:weaklusin}, we introduce the notion of a strongly regular measure and prove in Theorem~\ref{wLusin_equiv} that this condition is equivalent to the validity of the weak Lusin theorem in general topological spaces. Section~\ref{sect:stronglusin} is dedicated to proving Theorem~\ref{str_Lus_equiv}, which identifies optimal conditions for the strong Lusin theorem to hold, and in Section~\ref{sect:BR-BR}, we prove that Borel regularity is both necessary and sufficient for the existence of Borel representatives. 
Finally, in Section~\ref{sect:mainresults} we establish our main characterization  (Theorem~\ref{main}) and provide several examples illustrating the optimality of the assumptions in this result. 
\medskip

\noindent\textbf{Acknowledgments:} Part of this paper was written during PG's visit to Amherst College and PG wishes to thank
the institution for their hospitality.

\section{Borel regularity}
\label{sect:Borelreg}


We begin by recording a few measure-theoretic notions.
Let $(X,\tau)$ be a topological space and suppose that $\mathfrak{M}$ is any $\sigma$-algebra of subsets of $X$. A nonnegative measure $\mu:\mathfrak{M}\to[0,\infty]$ is said to be a \textit{Borel measure on $X$} provided ${\it Borel}_{\tau}(X)\subseteq\mathfrak{M}$, where ${\it Borel}_{\tau}(X)$ denotes the Borel $\sigma$-algebra on $(X,\tau)$, i.e., the smallest $\sigma$-algebra of $X$ containing $\tau$.
The quadruple $(X, \tau,\mathfrak{M},\mu)$ shall be called a \textit{topological measure space} if $(X, \tau)$ is a topological space, $\mathfrak{M}$ is a $\sigma$-algebra of subsets of $X$, and $\mu:\mathfrak{M}\to[0,\infty]$ is a nonnegative Borel measure on $X$. 
Recall that a Borel measure $\mu:\mathfrak{M}\to[0,\infty]$ is said to be
\emph{Borel regular}, if for every measurable set $E\in\mathfrak{M}$, there exists a Borel set $B\in{\it Borel}_{\tau}(X)$ such that $E\subseteq B$ and $\mu(B)=\mu(E)$.
We stress here that to be a Borel measure, we merely
demand that $\mathfrak{M}$ contains ${\textit Borel}_{\tau}(X)$
and not necessarily that 
$\mathfrak{M}={\textit Borel}_{\tau}(X)$. In the latter
case,  the measure $\mu$ would be automatically  Borel regular.

The following theorem is the main result of this section and it extends the work of \cite[Proposition~3.4.15]{MMM23} and \cite{AMS}.

\begin{tw}\label{in-out-reg}
Let $(X, \tau,\mathfrak{M},\mu)$ be a topological measure space with the property that there exists a collection $\{O_j\}_{j=1}^\infty$ of open sets satisfying
\begin{equation}\label{QLK-s-A2.bis-2}
X=\bigcup\limits_{j=1}^\infty O_j
\quad\mbox{and}\quad
\mu(O_j)<\infty\,\,\mbox{ for all }\,\,j\in{\mathbb{N}}.
\end{equation}
Then, the following conditions are equivalent.
\begin{enumerate}[(a)]
\item The measure $\mu$ is Borel regular, and every open set $U\subseteq X$ can be written as $U=F \cup Z$, where $F$ is an  $F_{\sigma}$ set and $Z$ is a null set.
\item The measure $\mu$ is outer regular on $\mathfrak{M}$, i.e., for  every measurable set $E\in\mathfrak{M}$,
\begin{equation}
\label{out-reg}
\mu(E) =\mathop{\inf}_{\substack{O- \text{open},\\ O\supseteq E}} \mu(O).
\end{equation}
\item The measure $\mu$ is inner regular on $\mathfrak{M}$, i.e.,  for  every  measurable set $E\in\mathfrak{M}$,
\begin{equation}
\label{in-reg}
\mu(E) =\mathop{\sup}_{\substack{C- \text{closed},\\ C\subseteq E}} \mu(C).
\end{equation}
\end{enumerate}
\end{tw}


\begin{remark}
\label{OR-BR}
It is easy to see that the outer regularity condition \eqref{out-reg} implies that the measure is Borel regular, even without the assumptions in \eqref{QLK-s-A2.bis-2}.
\end{remark}

\begin{remark} \label{inout}
In general, there is no relationship between the outer and inner  regularity conditions \eqref{out-reg} and \eqref{in-reg}. Indeed, from Example \ref{ex:wkL-1}, we have a topological measure space $(X,\tau,\mathfrak{M},\mu)$ such that the measure $\mu$ is inner regular but $\mu$ is not outer regular. On other hand, Example \ref{ex:Brep} gives us a topological measure space $(X,\tau,\mathfrak{M},\mu)$ such that the measure $\mu$ is outer regular but $\mu$ is not inner regular.
\end{remark}

To prove Theorem~\ref{in-out-reg}, we first prove  two lemmas. The first lemma is a slight generalization of \cite[Lemma~3.11]{AM15} (see also  \cite[Lemma~1.1]{EG15} for the case when $X=\mathbb{R}^n$ and \cite[Theorem~2.2.2]{Federer} for the case when $X$ is a metric space).

\begin{lem}
\label{finite-in-out-reg}
Let $(X, \tau,\mathfrak{M},\mu)$ be a topological measure space such that every open set $U\subseteq X$ can be written as $U=F \cup Z$, where $F$ is an  $F_{\sigma}$ set and $Z$ is a null set. Then, for every $B\in{\it Borel}_{\tau}(X)$ with $\mu(B)<\infty$ and every $\varepsilon\in(0,\infty)$, there exists a closed set $C_\varepsilon\subseteq X$ such that $C_\varepsilon\subseteq B$ and $\mu(B\setminus C_\varepsilon)<\varepsilon$.
\end{lem}
\begin{proof}
Let $B\in{\it Borel}_{\tau}(X)$ with $\mu(B)<\infty$. We define the set
\[
	\mathcal{F}:=\big\{A \in \mathfrak{M}: \forall\varepsilon\in(0,\infty), \, \exists C\subseteq X \, \text{such that $C$ is closed in $\tau$, } \, C\subseteq A, \, \text{ and }\, \mu\big(B\cap(A\setminus C)\big) <\varepsilon \big\}.
\]
Let us note that all closed sets and null sets belong to $\mathcal{F}$. Moreover, by arguing as in \cite[Lemma~3.11]{AM15}, we have that if $\{A_i\}_{i=1}^{\infty} \subseteq \mathcal{F}$, then the sets $\bigcup_{i=1}^{\infty}A_i, \bigcap_{i=1}^{\infty}A_i \in \mathcal{F}$. Therefore, by assumption, all open sets in $X$ belong to $\mathcal{F}$. Granted this, it is straightforward to check that the set
$\widetilde{\mathcal{F}}:=\big\{A\in\mathcal{F}:X\setminus A \in \mathcal{F}\big\}$
is a $\sigma$-algebra containing all open sets of $X$, and therefore $\widetilde{\mathcal{F}}$ contains $Borel_{\tau}(X)$. Thus, in particular, the Borel set $B$ belongs to $\mathcal{F}$ and the desired claim now follows. This finishes the proof of Lemma~\ref{finite-in-out-reg}.
\end{proof}

The next lemma extends part of \cite[Proposition~3.4.15]{MMM23}.
\begin{lem}
\label{Borel-in-out-reg}
Let $(X, \tau,\mathfrak{M},\mu)$ be a topological measure space with the property that there exists a collection $\{O_j\}_{j=1}^\infty$ of open sets  satisfying \eqref{QLK-s-A2.bis-2}.
%
%
Then, the following conditions are equivalent.
\begin{enumerate}[i)]
\item Every open set $U\subseteq X$ can be written as $U=F \cup Z$, where $F$ is an  $F_{\sigma}$ set and $Z$ is a null set.

\item For every $B\in{\it Borel}_{\tau}(X)$ and every $\varepsilon\in(0,\infty)$, there exists an open set $U_\varepsilon\subseteq X$ such that $U_\varepsilon\supseteq B$ and $\mu(U_\varepsilon\setminus B)<\varepsilon$.

\item For  every $B\in{\it Borel}_{\tau}(X)$ and every $\varepsilon\in(0,\infty)$, there exists a closed set $C_\varepsilon\subseteq X$ such that $C_\varepsilon\subseteq B$ and $\mu(B\setminus C_\varepsilon)<\varepsilon$.
\end{enumerate}
\end{lem}
\begin{proof}
We first prove that  \textit{i)} implies \textit{ii)}. 
Suppose \textit{i)} holds, and fix a Borel set $B\in{\it Borel}_{\tau}(X)$ along with a threshold $\varepsilon\in(0,\infty)$
Observe that each $O_j\setminus B\in Borel_{\tau}(X)$ with $\mu(O_j\setminus B)<\infty$. As such, for each $j\in\mathbb{N}$, we can appeal to Lemma~\ref{finite-in-out-reg} to get a closed set $C_j\subseteq O_j\setminus B$ satisfying
%
$\mu\big((O_j\setminus B)\setminus C_j\big)<\varepsilon/2^j.$
%
Consider the set
$$
U_\varepsilon:=\bigcup_{j=1}^\infty O_j\setminus C_j.
$$
Note that $U_\varepsilon\subseteq X$ is an open set and, since $C_j\subseteq X\setminus B$, we have that $O_j\cap B\subseteq O_j\setminus C_j$. Therefore,
$$
B=\bigcup_{j=1}^\infty(O_j\cap B)
\subseteq\bigcup_{j=1}^\infty
O_j\setminus C_j=U_\varepsilon
$$
and
\begin{align*}
\mu(U_\varepsilon\setminus B)&=
\mu\bigg(\bigcup_{j=1}^\infty(O_j\setminus C_j)\setminus B\bigg)
\leq
\sum_{j=1}^\infty\mu\big((O_j\setminus C_j)\setminus B\big)
\nonumber\\[6pt]
&=\sum_{j=1}^\infty
\mu\big((O_j\setminus B)\setminus C_j\big)
<\sum_{j=1}^\infty\varepsilon/2^j=\varepsilon.
\end{align*}
This completes the proof of \textit{i)} implies \textit{ii)}.

Next, let us suppose that $ii)$ holds. To prove \textit{iii)}, we let $B\in{\it Borel}_{\tau}(X)$ and $\varepsilon\in(0,\infty)$. From $ii)$, there exists an open set $U_\varepsilon\subseteq X$ such that $X\setminus B\subseteq U_\varepsilon$ and $\mu\big(U_\varepsilon\setminus (X\setminus B)\big)<\varepsilon$. Then $C_{\varepsilon}:= X\setminus U_\varepsilon$ is a closed set that is contained in $B$ and satisfies
\[
	\mu(B\setminus C_{\varepsilon}) = \mu(B \cap U_\varepsilon) = \mu\big(U_\varepsilon \setminus (X \setminus B) \big)< \varepsilon.
\] 
This proves statement \textit{iii)}.

Finally, we prove that \textit{iii)} implies \textit{i)}. Assume \textit{iii)} holds, and let $U\subseteq X$ be an open set. By $iii)$, for every $n\in\mathbb{N}$ there exists a closed set $C_n\subseteq X$ such that $C_n \subseteq U$ and $\mu(U \setminus C_n) < 1/n$. Let $C:=\bigcup_{n=1}^{\infty} C_n$. Then $C$ is an $F_{\sigma}$ set and we can write $U=C\cup(U\setminus C)$, where $\mu(U\setminus C)=0$. 
%
Thus $i)$ holds. This finishes the proof of the implication \textit{iii)} $\Rightarrow$ \textit{i)} and hence, the proof of Lemma~\ref{Borel-in-out-reg}.
\end{proof}

We are now ready to prove Theorem~\ref{in-out-reg}.
\begin{proof}[Proof of Theorem~\ref{in-out-reg}]
We begin proving that \textit{(a)} implies \textit{(b)}.
Let us suppose that \textit{(a)} holds, and let us fix $E \in \mathfrak{M}$. If $\mu(E)=\infty$ then \eqref{out-reg} trivially holds (by taking $O=X$) and so we can assume $\mu(E)<\infty$. With this in mind, let $\varepsilon\in(0,\infty)$. Since $\mu$ is Borel regular, there exists $B\in{\it Borel}_{\tau}(X)$ such that $E\subseteq B$ and $\mu(E)=\mu(B)$. By Lemma~\ref{Borel-in-out-reg}, there exists an open set $O\subseteq X$ such that $B\subseteq O$ and $\mu(O\setminus B) < \varepsilon$. By this, and the fact that $\mu(E)=\mu(B)<\infty$, we have
$\mu(O\setminus E) =\mu(O\setminus B)< \varepsilon.$
Given that $\varepsilon\in(0,\infty)$ was arbitrary, \eqref{out-reg} follows and hence, the implication \textit{(a)} $\Rightarrow$  \textit{(b)} holds.

Next, we prove that \textit{(b)} implies \textit{(c)}. Let us suppose that  \textit{(b)} holds.  
As a consequence of \eqref{QLK-s-A2.bis-2} and the outer regularity condition \eqref{out-reg}, for every measurable set $E\in \mathfrak{M}$, we have
\begin{equation}
\label{ej-4}
\inf_{\substack{C- \text{\textit{closed}},\\ C\subseteq E}} \mu(E \setminus C) =0.
\end{equation}
Indeed, \eqref{ej-4} follows from a more general set of results that we prove in Section~\ref{sect:weaklusin}, namely, Propositions~\ref{out_reg_sigma_fin} and \ref{inner=outer}. Noting that \eqref{ej-4} implies condition \eqref{in-reg} completes the proof of \textit{(b)} implies \textit{(c)}.

We will now show that \textit{(c)} implies \textit{(a)}. Let us suppose that  \textit{(c)} holds.  We fix any measurable set $E\in \mathfrak{M}$. By \eqref{QLK-s-A2.bis-2} for each $j\in \mathbb{N}$, we have $E\cap O_j\in\mathfrak{M}$ and $\mu(E\cap O_j)<\infty$. As such, by using  \eqref{in-reg} and arguing as in the proof of \textit{iii)} $\Rightarrow$ \textit{i)} in Lemma~\ref{Borel-in-out-reg}, for each $j\in \mathbb{N}$, we can find an  $F_{\sigma}$ set $F_{j}$ and a null set $Z_{j}$, such that 
$E\cap O_j = F_{j}\cup Z_j.$
Observe that the $F_{\sigma}$ set $F := \bigcup_{j=1}^{\infty} F_{j}$, and the null set $Z := \bigcup_{j=1}^{\infty} Z_{j}$ satisfy
\begin{align}\label{E=FuZ}
E = \bigcup_{j=1}^{\infty} E\cap O_j = \bigcup_{j=1}^{\infty} F_{j}\cup Z_j = \bigcup_{j=1}^{\infty} F_{j}\cup \bigcup_{j=1}^{\infty} Z_j = F \cup Z.
\end{align}
In particular, we have proved that  every open set $U\subseteq X$ can be written as $U=F \cup Z$, where $F$ is an  $F_{\sigma}$ set and $Z$ is a null set. 
Finally, we show that $\mu$ is Borel regular. For this purpose, we fix $E\in \mathfrak{M}$ and apply \eqref{E=FuZ} to the set $X\setminus E$ to obtain an  $F_{\sigma}$ set $F$ and a null set $Z$, such that $X\setminus E = F\cup Z$. Therefore, $E = (X\setminus F)\cap (X\setminus Z) \subseteq X\setminus F$, where $X\setminus F\in Borel_{\tau}(X)$.
Now, since  $(X\setminus F)\setminus E = (X\setminus E)\setminus F \subseteq Z$,
we obtain
$\mu((X\setminus F)\setminus E ) = 0$,
which implies $\mu(X\setminus F)=\mu(E)$. Thus, $\mu$ is Borel regular and this finishes the proof of \textit{(c)} $\Rightarrow$ \textit{(a)} and hence, the proof of Theorem~\ref{in-out-reg}.
\end{proof}

\section{Strongly regular measures and the weak Lusin theorem}
\label{sect:weaklusin}
The goal of this section is to identify optimal assumptions for the weak Lusin theorem to hold in general topological spaces (see Theorem~\ref{wLusin_equiv}). For this purpose, we introduce the notion of strongly regular measure.
\begin{defi}
\label{str-in}
Let $(X,\tau,\mathfrak{M},\mu)$ be a topological measure space. The measure $\mu$ is said to be  \emph{strongly regular} if, for every measurable set $E \in \mathfrak{M}$,
\begin{align*}
\inf_{\substack{C-\text{closed},\\O- \text{open},\\ C\subseteq E\subseteq O}} \mu(O \setminus C) =0.
\end{align*}
\end{defi}

\begin{prop} \label{inner=outer}
Let $(X,\tau,\mathfrak{M},\mu)$ be a topological measure space. 
Then, the following conditions are equivalent.
\begin{enumerate}
\item The measure $\mu$ is strongly regular.
\item For every measurable set $E \in \mathfrak{M}$,
\begin{align*}
\inf_{\substack{O- \text{open},\\ E\subseteq O}} \mu(O \setminus E) =0.
\end{align*}
\item For every measurable set $E \in \mathfrak{M}$,
\begin{align*}
\inf_{\substack{C- \text{closed},\\ C\subseteq E}} \mu(E \setminus C) =0.
\end{align*} 
\end{enumerate}
\end{prop}
\begin{proof}
The equivalence between \textit{(2)} and \textit{(3)}  follows from the identity $(X\setminus E) \setminus C = (X\setminus C) \setminus E$ and it is clear that \textit{(1)}  implies  \textit{(2)}. On the other hand, it is easy to see that \textit{(1)}  follows from \textit{(2)}  combined with \textit{(3)}. This completes the proof of Proposition~\ref{inner=outer}.
\end{proof}

We now collect a few basic properties and observations regarding strongly regular measures, including its relationship with the other measure regularity conditions appearing in Section~\ref{sect:Borelreg}. To begin, we have the following immediate consequence of Proposition~\ref{inner=outer} and Remark~\ref{OR-BR}.
\begin{cor}\label{str_reg_impl_Fsigma}
Let $(X,\tau,\mathfrak{M},\mu)$ be a topological measure space.
If $\mu$ is strongly regular, then every measurable set $E\in\mathfrak{M}$ can be written as $E= F\cup Z$ and $E = G\setminus N$, where $F$ is an $F_{\sigma}$ set, $G$ is a $G_{\delta}$ set, and $Z,N$ are null sets. Moreover, $\mu$ is both inner regular and outer regular on $\mathfrak{M}$. In particular, $\mu$ is Borel regular.
\end{cor}

From Corollary~\ref{str_reg_impl_Fsigma} and Remark~\ref{OR-BR}, we have that
\[\mbox{strong regularity }\Longrightarrow\mbox{ outer regularity}\Longrightarrow\mbox{ Borel regularity.}
\]
The converses of these implications fail, in general; see Examples~\ref{ex:wkL-1}, \ref{weakl}, and Remark \ref{r52}.  Furthermore, inner regularity does not imply strong regularity; see Remark~\ref{inout}.  However, as the next two results highlight, some of these conditions are equivalent under certain mild assumptions on the topological measure space.
\begin{prop}\label{out_reg_sigma_fin}
Let $(X,\tau,\mathfrak{M},\mu)$ be a topological measure space. If $\mu$ is $\sigma$-finite and outer regular on $\mathfrak{M}$, then $\mu$ is strongly regular.
\end{prop}
Note that the conclusion of Proposition~\ref{out_reg_sigma_fin} can fail if the measure is not $\sigma$-finite or if the measure is simply Borel regular; see Remark~\ref{r52} and Example~\ref{ex:wkL-1}. 
\begin{proof}
Suppose that $\mu$ is $\sigma$-finite and outer regular on $\mathfrak{M}$. Then we can find
a sequence $\{A_{j}\}_{j=1}^\infty$ of sets having finite measure such that $X= \bigcup_{j=1}^\infty A_{j}$. Let $E\in\mathfrak{M}$ and $\varepsilon\in(0,\infty)$. For every $j\in \mathbb{N}$, the measure of $A_{j}\cap E$ is finite and so by the outer regularity condition \eqref{out-reg}, we can find a sequence $\{O_{j}\}_{j=1}^\infty$ of open sets such that $A_{j}\cap E \subseteq O_{j}$ and 
$\mu(O_{j} \setminus (A_{j}\cap E)) < \varepsilon/2^j$ for every $j\in \mathbb{N}$. Since $X=\bigcup_{j=1}^\infty A_j$, the open set $O := \bigcup_{j=1}^\infty O_{j}$ contains $E$ and satisfies
\begin{align*}
\mu(O \setminus E) = \mu\bigg( \bigcup_{j=1}^\infty O_{j} \setminus E \bigg) \leq \sum_{j=1}^\infty \mu(O_{j} \setminus E) \leq \sum_{j=1}^\infty \mu\big(O_{j} \setminus (A_{j}\cap E)\big) < \varepsilon.
\end{align*}
Thus, condition \textit{(2)} in Proposition~\ref{inner=outer} holds and so the strong regularity of $\mu$ now follows from \textit{(1)} in Proposition~\ref{inner=outer}.
\end{proof}

The next result follows immediately from Theorem~\ref{in-out-reg} and Proposition~\ref{out_reg_sigma_fin}.
\begin{cor}
\label{Borel-in-out-reg-2}
Let $(X, \tau,\mathfrak{M},\mu)$ be a topological measure space with the property that there exists a collection $\{O_j\}_{j=1}^\infty$ of open sets  satisfying \eqref{QLK-s-A2.bis-2}, and suppose that  every open set $U\subseteq X$ can be written as $U=F \cup Z$, where $F$ is an  $F_{\sigma}$ set and $Z$ is a null set.  If $\mu$ is Borel regular (equivalently, if $\mu$ is either outer or inner regular on $\mathfrak{M}$), then $\mu$ is strongly regular. Consequently, every Borel regular measure on a metric space is strongly regular.\footnote{We suppose that every ball has finite measure.}
\end{cor}
We remark that the conclusion of Corollary~\ref{Borel-in-out-reg-2} may not hold if either of the assumptions on the topological measure space are omitted; see Examples~\ref{ex:wkL-1} and \ref{weakl}.

The following characterization of the weak Lusin theorem is the main result of this section.
\begin{tw} \label{wLusin_equiv}
Let $(X,\tau,\mathfrak{M},\mu)$ be a topological measure space. Then, the following statements are equivalent.
\begin{enumerate}[(1)]
\item The measure $\mu$ is strongly regular. 
\item The weak Lusin theorem holds, i.e., for every measurable function $u:X\to\overline{\mathbb{R}}$ and any $\varepsilon\in(0,\infty)$, there exists closed set $C_{\varepsilon}\subseteq X$ such that 
$\mu(X\setminus C_{\varepsilon}) < \varepsilon$ and restriction of $u$ to $C_{\varepsilon}$ is continuous with respect to $\tau\vert_{C_{\varepsilon}}$.
\end{enumerate}
\end{tw}
\begin{proof}
Assume first that $\mu$ is strongly regular and let $\{V_{n}\}_{n=1}^{\infty}$ be a countable basis for the topology of $\overline{\mathbb{R}}$. For each $n\in\mathbb{N}$, let $E_n := u^{-1}(V_{n})\subseteq X$. We fix $\varepsilon\in (0,\infty)$ and $n\in \mathbb{N}$. Using the strong regularity of $\mu$, we can find an open set $O_n\subseteq X$ and a closed set $D_{n}\subseteq X$ such that $D_{n}\subseteq E_{n} \subseteq O_{n}$ and
$\mu(O_{n}\setminus D_{n}) < \varepsilon/2^{n}.$
We define 
\begin{align*}
F_{\varepsilon} := \left(\bigcup_{n=1}^{\infty} O_{n}\setminus D_{n}\right)  \quad \text{and} \quad C_{\varepsilon} := X\setminus F_{\varepsilon}.
\end{align*}
Observe that $C_{\varepsilon}\subseteq X$ is a closed set and 
\[
\mu(X\setminus C_{\varepsilon})=\mu(F_{\varepsilon})\leq\sum_{n=1}^{\infty}\mu(O_{n}\setminus D_{n}) < \varepsilon.
\] 
We shall now show that the function $f := u\vert_{C_{\varepsilon}}$ is continuous. Observe that for every $n\in \mathbb{N}$, we have 
\begin{align*}
 O_{n}\cap C_{\varepsilon} =  E_{n}\cap C_{\varepsilon} = f^{-1}(V_n).
\end{align*}
Thus, $f^{-1}(V_n) =  O_{n}\cap C_{\varepsilon}$ is an open set in the subspace topology on $C_{\varepsilon}$. Since $\{V_{n} \}_{n=1}^{\infty}$ is the basis for the topology of $\overline{\mathbb{R}}$, it follows that $f$ is continuous. Hence, the weak Lusin theorem holds.

Suppose next that \textit{(2)} holds. To show that $\mu$ is strongly regular, let  $E\in\mathfrak{M}$ be any measurable set. By \textit{(2)}, for every $n\in \mathbb{N}$ there exists a closed set $F_{n}\subseteq X$ such that $\mu(X\setminus F_{n}) < 1/n$ and the function  $u_{n} := {\chi_{E}}\vert_{F_{n}}$ is continuous. Hence, $u_{n}^{-1}(\{1\}) = E\cap F_{n} $ is closed in the topology $\tau\vert_{F_{n}}$ and so there exists a set $C_{n}\subseteq X$, which is closed in topology $\tau$ and satisfies $E\cap F_{n} = C_{n} \cap F_{n}$. Therefore,
the set $C_{n}\cap F_{n}\subseteq E$ is closed in $\tau$. Moreover, since for every $n\in\mathbb{N}$,
\begin{align*}
\mu\big(E\setminus (C_{n}\cap F_{n})\big)=\mu(E\setminus F_{n}) \leq \mu(X\setminus F_{n}) < 1/n,
\end{align*}
it follows that 
\begin{align*}
\inf_{\substack{C- \text{\textit{closed}},\\ C\subseteq E}} \mu(E \setminus C) =0.
\end{align*}
Thus, condition \textit{(3)} in Proposition~\ref{inner=outer} holds and so the strong regularity of $\mu$ now follows from \textit{(1)} in Proposition~\ref{inner=outer}. This completes the proof of Theorem~\ref{wLusin_equiv}.
\end{proof}

A slight modification of the proof of Theorem~\ref{wLusin_equiv} gives the following conclusion.

\begin{prop}\label{weak_Lus_Bor}
Let $(X,\tau,\mathfrak{M},\mu)$ be a topological measure space.
%
Suppose that for every measurable function $u:X\to\overline{\mathbb{R}}$ and for every $\varepsilon\in(0,\infty)$, there exists a  set $B_{\varepsilon}\in{\it Borel}_{\tau}(X)$ such that 
$\mu(X\setminus B_{\varepsilon}) < \varepsilon$ and the restriction of $u$ to $B_{\varepsilon}$ is continuous with respect to $\tau\vert_{B_{\varepsilon}}$.
Then $\mu$ is Borel regular.
\end{prop}
Let us remark that we will see in Example~\ref{weakl}, that the reverse implication in Proposition~\ref{weak_Lus_Bor} does not hold, in general.
\begin{proof}
Let $E\in\mathfrak{M}$ be any measurable set. Under the current assumptions, for every $n\in \mathbb{N}$, there is a Borel set $B_{n}\subseteq X$ such that $\mu(X\setminus B_{n}) < 1/n$ and the function $u_{n} := \chi_{E}\vert_{B_{n}}$ is continuous.  Hence, $u_{n}^{-1}(\{1\}) = E\cap B_{n}$ is closed in topology $\tau\vert_{B_{n}}$ and so there exists a set $C_{n}\subseteq X$, which is closed in topology $\tau$ and satisfies $E\cap B_{n} = C_{n} \cap B_{n}$. Observe that for each $n\in \mathbb{N}$,
\begin{align*}
E = \big(E\cap B_{n}) \cup \big(E \setminus B_{n} \big)=  \big(C_{n}\cap B_{n}) \cup \big( E \setminus B_{n} \big) \subseteq \big(C_{n}\cap B_{n}) \cup \big( X \setminus B_{n} \big).
\end{align*}
For each  $n\in \mathbb{N}$, let $U_n := \big(C_{n}\cap B_{n}) \cup \big( X \setminus B_{n} \big) \in Borel_{\tau}(X)$. Then $E\subseteq U_n$ and for each $n\in \mathbb{N}$,
$\mu(U_n \setminus E) \leq  \mu( X \setminus B_{n}) < 1/n$.
Thus,
\begin{align*}
\inf_{\substack{U- \text{\textit{Borel}},\\ E\subseteq U}} \mu(U \setminus E) =0,
\end{align*}
which further implies that $\mu$ is  Borel regular. This completes the proof of Proposition~\ref{weak_Lus_Bor}.
\end{proof}

\section{Almost normal spaces and the strong Lusin theorem}
\label{sect:stronglusin}

In this section we identify necessary and sufficient conditions for the strong Lusin theorem to hold (see Theorem~\ref{str_Lus_equiv}).
The proof of the following result can be found in \cite[Theorem 15.8]{Willard}.

\begin{lem}[Tietze extension theorem]\label{tietze}
Given a topological space $(X,\tau)$, the following statements are equivalent.
\begin{enumerate}[$a)$]
\item The topological space $(X,\tau)$ is normal, in the sense that if $A,B\subseteq X$ are two disjoint closed sets, then there exist two disjoint open sets $U,V\subseteq X$ such that $A\subseteq U$ and $B\subseteq V$.

\item If $C\subseteq X$ is closed in $\tau$ and $f:C\to\overline{\mathbb{R}}$ 
is a continuous function with respect to $\tau\vert_C$, then there exists a function
$F:X\to\overline{\mathbb{R}}$, which is continuous with respect to $\tau$ and satisfies $F=f$ pointwise in $C$.
\end{enumerate}
\end{lem}

In the spirit of the characterization in Lemma~\ref{tietze}, we introduce the following condition.

\begin{defi}\label{almost_normal}
A topological measure space $(X,\tau,\mathfrak{M},\mu)$ is said to be \emph{almost normal} if, for every set $C\subseteq X$ that is closed in $\tau$, any function $f:C\to\overline{\mathbb{R}}$ that is continuous
with respect to $\tau\vert_C$, and any $\varepsilon \in (0,\infty)$, there exist both a set $C_{\varepsilon} \subseteq C$, which is closed in $\tau$, and a function $F_{\varepsilon}:X\to\overline{\mathbb{R}}$,  which is continuous with respect to $\tau$, satisfying $\mu(C\setminus C_{\varepsilon}) < \varepsilon$ and $F_{\varepsilon}=f$ pointwise in $C_{\varepsilon}$.
\end{defi}

In the next result, we construct a bridge between normality and almost normality.
\begin{prop}\label{norm_almost_Tietze}
Let $(X,\tau)$ be a topological space. Then, the following statements are equivalent.
\begin{enumerate}[(1)]
\item The topological space $(X,\tau)$ is normal.
\item For every Borel measure $\mu:\mathfrak{M}\to[0,\infty]$,  the topological measure space $(X,\tau, \mathfrak{M}, \mu)$ is almost normal.
\end{enumerate}
\end{prop}
\begin{proof}
We only have to show \textit{(2)} implies \textit{(1)}, since the implication \textit{(1)} $\Rightarrow$ \textit{(2)} follows from Tietze extension theorem (Lemma \ref{tietze}).

Assume $\textit{(2)}$ holds and let $A, B$ be two closed disjoint subsets of $X$. Let $\mathfrak{M}:= Borel_{\tau}(X)$ and consider the Borel measure $\mu : \mathfrak{M} \rightarrow [0,\infty]$, that is given by $\mu(\emptyset):=0$ and $\mu(E):=\infty$ for nonempty $E \in \mathfrak{M}$. Moreover, we consider the function $f: A\cup B \rightarrow \overline{\mathbb{R}}$, defined by $f:= \chi_{A}$, which is continuous with respect to $\tau\vert_{A \cup B}$. Since $(X, \tau, \mathfrak{M}, \mu)$ is almost normal, there exist both a set $C \subseteq A \cup B$, which is closed in $\tau$, and a continuous function $g: X \to \overline{\mathbb{R}}$, such that $\mu((A \cup B) \setminus C) < 1$ and $g = f$ pointwise in $C$. From the definition of $\mu$, we have $A\cup B = C$. Hence, $g$ is a continuous function satisfying $g\vert_{A} = 1$ and $g\vert_{B} = 0$. Consequently, the open sets $U:=g^{-1}((1/2,\infty))\subseteq X$ and $V:=g^{-1}((-\infty,1/2))\subseteq X$ are disjoint and satisfy $A\subseteq U$ and $B\subseteq V$. Thus, $(X,\tau)$ is normal, as wanted. 
\end{proof}

%

We are now ready to present the main result of this section.

\begin{tw}\label{str_Lus_equiv}
Let $(X,\tau,\mathfrak{M},\mu)$ be a topological measure space. Then, the following statements are equivalent.
\begin{enumerate}[(1)]
\item The topological  measure space $(X,\tau,\mathfrak{M},\mu)$ is almost normal and $\mu$ is strongly regular.
\item The strong Lusin theorem holds, i.e., for every measurable function $u:X\to\overline{\mathbb{R}}$ and any $\varepsilon\in(0,\infty)$, there exists a closed set $C_{\varepsilon}\subseteq X$ and a continuous function $g:X\to\overline{\mathbb{R}}$ such that 
$\mu(X\setminus C_{\varepsilon}) < \varepsilon$ and $g=u$ pointwise in $C_{\varepsilon}$.
\end{enumerate}
\end{tw}
\begin{proof}
Assume first that \textit{(2)} holds. Since the strong Lusin theorem implies the weak Lusin theorem, we have that $\mu$ is strongly regular by Theorem~\ref{wLusin_equiv}. Thus, we only have to check almost normality of $(X,\tau,\mathfrak{M},\mu)$. Let $C\subseteq X$ be any set that is closed in $\tau$,  $f:C\to \overline{\mathbb{R}}$ be any function that is continuous in $\tau\vert_{C}$, and $\varepsilon\in (0,\infty)$ by any fixed number. We can extend $f$ to measurable function $\tilde{f}:X\to \overline{\mathbb{R}}$ by setting $\tilde{f}=0$ outside of $C$. By applying the strong Lusin theorem to $\tilde{f}$, we obtain a closed set $C_{\varepsilon}\subseteq X$ (with respect to $\tau$) and a continuous function $F: X \to \overline{\mathbb{R}}$ (with respect to $\tau$) such that $\mu(X\setminus C_{\varepsilon})<\varepsilon$ and $F = \tilde{f}$ pointiwse in $C_{\varepsilon}$.
Consider the set $\tilde{C}_{\varepsilon} := C\cap C_{\varepsilon}$, which is closed in $\tau$. Since $\tilde{f} = f$ on $C$ and $\mu(C\setminus \tilde{C}_{\varepsilon})  =  \mu(C\setminus C_{\varepsilon}) \leq \mu(X\setminus C_{\varepsilon}) < \varepsilon$, we conclude that $(X,\tau,\mathfrak{M},\mu)$  is almost normal. Thus, \textit{(1)} holds, as wanted.

Suppose next that \textit{(1)} holds. Let $u:X \to \overline{\mathbb{R}}$ be a measurable function. Since $\mu$ is strongly regular, by Theorem~\ref{wLusin_equiv}, for every $\varepsilon\in (0,\infty)$, we can find closed set $C_{\varepsilon} \subseteq X$ (with respect to $\tau$) such that $\mu(X\setminus C_{\varepsilon}) < \varepsilon/2$ and $u\vert_{C_{\varepsilon}}$ is continuous with respect to $\tau\vert_{C_{\varepsilon}}$. From almost normality of $(X,\tau,\mathfrak{M},\mu)$, we can find a closed set $\tilde{C}_{\varepsilon} \subseteq C_{\varepsilon}$ (with respect to $\tau$) and a continuous function $F: X \to \overline{\mathbb{R}}$ (with respect to $\tau$), such that $\mu(C_{\varepsilon} \setminus \tilde{C}_{\varepsilon}) < \varepsilon/2$ and $F = u\vert_{C_{\varepsilon}}$ pointwise in $\tilde{C}_{\varepsilon}$.
Therefore, $F = u$ pointwise in $\tilde{C}_{\varepsilon}$ and $$\mu(X \setminus \tilde{C}_{\varepsilon}) = \mu(X \setminus {C}_{\varepsilon}) + \mu({C}_{\varepsilon} \setminus \tilde{C}_{\varepsilon})  < \varepsilon.$$
Hence, \textit{(2)} holds and this completes the proof of Theorem~\ref{str_Lus_equiv}.
\end{proof}

Combining Theorems~\ref{str_Lus_equiv} and \ref{wLusin_equiv}, we obtain the following conclusion.

\begin{cor}\label{cor_weak_str}
If a topological measure space $(X,\tau,\mathfrak{M},\mu)$ is almost normal, then the weak Lusin theorem implies the strong Lusin theorem.
\end{cor}

\section{The existence of Borel representatives}
\label{sect:BR-BR}
The main goal of this section is to prove an equivalence between Borel regularity and the existence of Borel representatives. More specifically, we prove the following theorem.

\begin{tw}\label{rep-reg}
Let $(X,\tau,\mathfrak{M},\mu)$ be a topological measure space where the measure $\mu$ is $\sigma$-finite.
Then, the following statements are equivalent.
\begin{enumerate}[(1)]
\item The measure $\mu$ is Borel regular.
\item For every measurable function $u:X\to\overline{\mathbb{R}}$, there exist a Borel set $N\subseteq X$ and a Borel function $f:X\to\overline{\mathbb{R}}$ such that 
$\mu(N)=0$ and $f=u$ pointwise in $X\setminus N$.
\end{enumerate}
\end{tw}
We remark here that the implication \textit{(1)} $\Rightarrow$ \textit{(2)} is a slight refinement of a known result; see, for example, \cite[Proposition~3.3.23]{HKST} and \cite[p.77]{Federer}. Moreover, as the proof of Theorem~\ref{rep-reg} will reveal, the implication \textit{(2)} $\Rightarrow$ \textit{(1)} is true, even if $\mu$ is not $\sigma$-finite.
\begin{proof}
Suppose first that $\mu$ is Borel regular.
Since $\mu$ is $\sigma$-finite, there exists  a family $\{E_{n}\}_{n=1}^\infty$  of measurable sets having finite measure such that $\bigcup_{n=1}^\infty E_{n} = X$.  
We can assume that $E_n \subseteq E_{n+1}$ for every $n \in \mathbb{N}$. Let $u:X\to\overline{\mathbb{R}}$ be a measurable function. Then, we can find a sequence $\{g_n\}_{n=1}^\infty$ of simple functions converging to $u$ pointwise in $X$. Next, for each $n \in \mathbb{N}$, we define $f_n:=g_n \chi_{E_n}$. In this way, we have constructed a sequence $\{f_n\}_{n=1}^\infty$ of simple functions that converges to $u$ pointwise in $X$ and satisfies $\mu(\{x \in X: f_n(x) \neq 0\})< \infty$, for all $n \in \mathbb{N}$. Let us fix $n \in \mathbb{N}$. By the Borel regularity of $\mu$, for every $y \in f_n (X)\cap (\overline{\mathbb{R}} \setminus \{0\})$, there exists a Borel set $B^n_y\in Borel_{\tau}(X)$ such that $f_n^{-1}(\{y\}) \subseteq B^n_y$ and $\mu(f_n^{-1}(\{y\}))= \mu( B^n_y)$. Let 
\[
	b_n:=\sum_{y \in f_n (X)\cap (\overline{\mathbb{R}} \setminus \{0\})} y \chi_{B^n_y}.
\]
Since $f_n (X)$ is a finite set, it is easy to see that $b_n:X\to\overline{\mathbb{R}}$ is a Borel function such that $b_n=f_n$ pointwise a.e. in $X$. By this, and the Borel regularity of $\mu$, we can find a  null set $N_n\in Borel_{\tau}(X)$ such that $b_n =f_n$ pointwise in $X \setminus N_n$. Hence, if we define $N:=\bigcup_{n=1}^\infty N_n$, then we have $b_n=f_n$ pointwise in $X\setminus N$ and so $\{b_n\}_{n=1}^\infty$  converges to $u$ pointwise in $X \setminus N$. Therefore, we have that $u\vert_{X\setminus N}$ is a Borel function (in the topology $\tau\vert_{X\setminus N}$). Now, we extend $u\vert_{X\setminus N}$ by zero to a function $f:X\to\overline{\mathbb{R}}$. Then $f=u$ in $X\setminus N$, where $N \in Borel_\tau(X)$ and $\mu(N)=0$. We claim that $f$ is a Borel function. Suppose that $I\subseteq\overline{\mathbb{R}}$ is an open set. Then $\big(u\vert_{X\setminus N}\big)^{-1}(I)$ is a Borel set and so there exists $B\in Borel_\tau(X)$ such that $\big(u\vert_{X\setminus N}\big)^{-1}(I)=B\cap (X\setminus N)$. If $0\in I$ then
$$
f^{-1}(I)=N\cup\big(u\vert_{X\setminus N}\big)^{-1}(I)=N\cup[B\cap (X\setminus N)]\in Borel_\tau(X),
$$
and if $0\not\in I$ then 
$$
f^{-1}(I)=\big(u\vert_{X\setminus N}\big)^{-1}(I)=B\cap (X\setminus N)\in Borel_\tau(X).
$$
Hence, $f$ is a Borel function, as wanted.

Assume that \textit{(2)} holds. To prove \textit{(1)}, let $A\subseteq X$ be a measurable set. If $\mu(A)=\infty$ then $X\in Borel_\tau(X)$ satisfies $\mu(A)=\mu(X)$. Thus, we can assume that $\mu(A)<\infty$. As such, we can appeal to \textit{(2)} to get a Borel set $N\subseteq X$ and a Borel function $f:X\to\overline{\mathbb{R}}$ such that 
$\mu(N)=0$ and $\chi_A=f$ pointwise in $X\setminus N$. Let $B:=f^{-1}(\{1\})$. Then $B$ is a Borel set such that $B=(A\setminus N)\cup (B\cap N)$. Thus, $\mu\big(B\big)=\mu(A\setminus N)=\mu(A)$ and so $N\cup B$ is a Borel set satisfying $A\subseteq N\cup B$ and $\mu\big(N \cup B\big)=\mu(A)$. This completes the proof of Theorem~\ref{rep-reg}.
\end{proof}

The next example shows that the implication  \textit{(1)} $\Rightarrow$ \textit{(2)}  in Theorem~\ref{rep-reg} may fail to hold if the measure is not $\sigma$-finite.

\begin{ex}
\label{ex:Brep}
There exists a topological measure space $(X,\tau,\mathfrak{M},\mu)$ such that the measure $\mu$ is outer regular (hence, also Borel regular) but not $\sigma$-finite, and condition \textit{(2)} in Theorem~\ref{rep-reg} (existence of Borel representatives) fails to hold.
\end{ex}
\begin{proof}
Let $X:=\{a,b,c\}$ and let $\mathfrak{M}:=2^X$. Consider the topology $\tau:=\big\{\varnothing,X,\{c\}\big\}$, and define a measure $\mu:\mathfrak{M}\to[0,\infty]$, by setting $\mu(\varnothing):=0$, $\mu(\{a\}):=\infty$, $\mu(\{b\}):=\infty$, and $\mu(\{c\}):=1$. In this case, it is easy to check that $\mu$ is not $\sigma$-finite, but is outer regular, and hence Borel regular by Remark~\ref{OR-BR}. Moreover, since $Borel_{\tau}(X)=\big\{\varnothing,X,\{c\},\{a,b\}\big\}$, one can verify that the measurable function $u:=\chi_{\{a\}}$ does not have a Borel representative in the sense described in statement \textit{(2)} in Theorem~\ref{rep-reg}.
\end{proof}
\begin{remark} \label{r52}
Note that the measure $\mu$, defined in the proof of Example~\ref{ex:Brep}, is not strongly regular and so we see that Proposition~\ref{out_reg_sigma_fin} may fail to hold if the measure is not $\sigma$-finite.
\end{remark}
\section{Main result}
\label{sect:mainresults}

In this section, we present the principal result of this article (Theorem~\ref{main}) and provide several examples that illustrate the optimality of the conditions appearing in the formulation of this theorem. 


\begin{tw}\label{main}
Let $(X,\tau,\mathfrak{M},\mu)$ be a topological measure space with the property that there exists a collection $\{O_j\}_{j=1}^\infty$ of open sets satisfying
\begin{equation}\label{OSF-main}
X=\bigcup\limits_{j=1}^\infty O_j
\quad\mbox{and}\quad
\mu(O_j)<\infty\,\,\mbox{ for all }\,\,j\in{\mathbb{N}}.
\end{equation}
Suppose further that every open set $U\subseteq X$ can be written as $U=F \cup Z$, where $F$ is an $F_{\sigma}$ set and  $Z$ is a  null set. Then, the following statements are equivalent. 
\begin{enumerate}[(1)]
\item The measure $\mu$ is Borel regular.

\item Each measurable function has a Borel representative, i.e., for every measurable function $u:X\to\overline{\mathbb{R}}$, there exists a Borel set $N\subseteq X$ and a Borel function $f:X\to\overline{\mathbb{R}}$ such that 
$\mu(N)=0$ and $f=u$ pointwise in $X\setminus N$.

\item The weak Lusin theorem holds, i.e., for every measurable function $u:X\to\overline{\mathbb{R}}$ and any $\varepsilon\in(0,\infty)$, there exists closed set $C_{\varepsilon}\subseteq X$ such that 
$\mu(X\setminus C_{\varepsilon}) < \varepsilon$ and restriction of $u$ to $C_{\varepsilon}$ is continuous with respect to $\tau\vert_{C_{\varepsilon}}$.

\item For every measurable function $u:X\to\overline{\mathbb{R}}$ and any $\varepsilon\in(0,\infty)$, there exists a Borel set $B_{\varepsilon}\subseteq X$ such that $\mu(X\setminus B_{\varepsilon}) < \varepsilon$ and restriction of $u$ to $B_{\varepsilon}$ is continuous.
\end{enumerate}
If, in addition, $(X,\tau)$ is assumed to be normal then each of the statements above are further equivalent to the following statement.
\begin{enumerate}[(1)]\addtocounter{enumi}{4}
\item The strong Lusin theorem holds, i.e., for every measurable function $u:X\to\overline{\mathbb{R}}$ and any $\varepsilon\in(0,\infty)$, there exists a closed set $C_{\varepsilon}\subseteq X$ and a continuous function $g:X\to\overline{\mathbb{R}}$ such that 
$\mu(X\setminus C_{\varepsilon}) < \varepsilon$ and $g=u$ pointwise in $C_{\varepsilon}$.
\item For every measurable function $u:X\to\overline{\mathbb{R}}$ and any $\varepsilon\in(0,\infty)$, there exists a Borel set $B_{\varepsilon}\subseteq X$ and a continuous function $g:X\to\overline{\mathbb{R}}$ such that  $\mu(X\setminus B_{\varepsilon}) < \varepsilon$ and $g=u$ pointwise in $B_{\varepsilon}$.
\end{enumerate}
\end{tw}

Let us remark that the implication \textit{(1)} $\Rightarrow$ \textit{(3)} in Theorem~\ref{main} was proved in \cite{AGS} with the condition in \eqref{OSF-main} replaced by the stronger assumption that $\mu(X)<\infty$.

\begin{proof}
The implications in this theorem follow from results previously proven in this paper. More specifically:
\begin{itemize}
\item $\textit{(1)} \Leftrightarrow \textit{(2)}$ by Theorem~\ref{rep-reg};
\item $\textit{(1)} \Rightarrow \textit{(3)}$ by Theorem~\ref{in-out-reg}, Proposition~\ref{out_reg_sigma_fin}, and Theorem~\ref{wLusin_equiv};
\item $\textit{(3)} \Rightarrow \textit{(4)}$ is obvious;
\item $\textit{(4)} \Rightarrow \textit{(1)}$ by Proposition~\ref{weak_Lus_Bor}.
\end{itemize}
Under normality of $(X,\tau)$, we have that:
\begin{itemize}
\item $\textit{(5)} \Rightarrow \textit{(6)}$ is obvious;
\item $\textit{(6)} \Rightarrow \textit{(4)}$ is obvious;
\item $\textit{(3)} \Rightarrow \textit{(5)}$ by Theorem~\ref{wLusin_equiv}, Proposition~\ref{norm_almost_Tietze}, and Theorem~\ref{str_Lus_equiv}.
\end{itemize}
This completes the proof of Theorem~\ref{main}.
\end{proof}

As the following example illustrates, the equivalence between \textit{(5)} and \textit{(1)} - \textit{(4)}  in Theorem~\ref{main} may fail if the space $(X,\tau)$ is not normal.
\begin{ex}
There exists a topological measure space $(X,\tau,\mathfrak{M},\mu)$ such that $(X,\tau)$ is not normal, \eqref{OSF-main} holds, every open set can be written as the union of an $F_\sigma$ set and a null set,  and for which \textit{(3)} in Theorem~\ref{main} (the weak Lusin theorem) holds and  \textit{(5)} in Theorem~\ref{main} (the strong Lusin theorem) fails.
\end{ex}

\begin{proof}
Let $X:=\{a,b,c\}$. Consider the topology $\tau:=\big\{\varnothing,X,\{c\},\{a,c\},\{b,c\}\big\}$, and define a measure $\mu:\mathfrak{M}\to[0,\infty]$, by setting $\mu := \delta_a + \delta_b$, where $\mathfrak{M}:=2^X$ and $\delta_{x}$ is the Dirac measure concentrated at $x\in X$. In this case, $\mu(X)=2<\infty$ and so $(X,\tau,\mathfrak{M},\mu)$ satisfies \eqref{OSF-main}.
Moreover, $\mu$ is trivially Borel regular since $Borel_{\tau}(X)=2^X$ and it is easy to check that every open set can be written as the union of an $F_\sigma$ set and a null set. Thus, \textit{(1)} - \textit{(3)} in Theorem~\ref{main} hold. 

We want to show that \textit{(5)} in Theorem~\ref{main} fails. Consider the function $u: X\to\mathbb{R}$ defined by setting $u(a):=0$, $u(b):=1$, and $u(c):=2$.
%
This function is trivially measurable. Now let $\varepsilon\in(0,1)$. Then from the definition of $\mu$, we have that $C_\varepsilon=X$ and $C_\varepsilon=\{a,b\}$ are the only closed sets satisfying $\mu(X\setminus C_\varepsilon)<\varepsilon$. Moreover, if $g:X\to\overline{\mathbb{R}}$ is continuous then $g$ must be a constant function and so we immediately conclude that $g\neq u$ pointwise in $C_\varepsilon$. Hence, \textit{(5)} in Theorem~\ref{main} fails, as wanted.
\end{proof}

In the next example, we see that it is possible for the strong Lusin theorem to hold in the absence of normality.
\begin{ex}
There exists a topological measure space $(X,\tau,\mathfrak{M},\mu)$ such that $(X,\tau)$ is not normal, $\mu$ is not trivial, and \textit{(5)} in Theorem~\ref{main} (the strong Lusin theorem) holds.
\end{ex}
\begin{proof}
Let $X:=\{a,b,c\}$. Consider the topology $\tau:=\big\{\varnothing,X,\{c\},\{a,c\},\{b,c\}\big\}$, and define a measure $\mu:\mathfrak{M}\to\mathbb{R}$, by setting $\mu :=\delta_a$, where $\mathfrak{M}:=2^X$. It is easy to check that the measure $\mu$ is strongly regular  since $\{a\}$ is closed and $\{b,c\}$ is null set. We now show that $(X,\tau,\mathfrak{M},\mu)$ is almost normal. It is clear, that for closed sets $ \emptyset, X , \{b\}, \{a\} $ and continuous functions defined on them, we can find continuous extensions defined on whole space. The last closed set that we need to consider is $\{a,b\}$. Let $C:=\{a,b\}$ and $f:C \to \overline{\mathbb{R}}$ be any continuous function with respect to $\tau\vert_{C}$. Given any $\varepsilon\in (0,\infty)$, we take the closed set $C_{\varepsilon}:= \{a\}$. Then $f\vert_{C_{\varepsilon}}$ is now a constant function and $F_{\varepsilon}: X \to \overline{\mathbb{R}}$ given by $F_{\varepsilon}(x) := f(a)$, for all $x\in X$, is a continuous extension of $f\vert_{C_{\varepsilon}}$. From these observations, we have that \textit{(5)} in Theorem~\ref{main}  holds by Theorem~\ref{str_Lus_equiv}.
\end{proof}




In the following example, we see that \textit{(3)} in Theorem~\ref{main} (the weak Lusin theorem) can hold in the absence of \eqref{OSF-main}.

\begin{ex}
Consider the topological measure space $(\mathbb{R},d,\mathfrak{M},\mu)$, where $d$ is the discrete metric and $\mu$ is the counting measure defined on $\mathfrak{M}:=2^{\mathbb{R}}$. Then $\mu$ is a Borel measure that is strongly regular but not $\sigma$-finite. In particular, \textit{(3)} in Theorem~\ref{main} (the weak Lusin theorem) holds but condition \eqref{OSF-main} does not hold.
\end{ex}
\begin{proof}
Since every set is closed, for every $E\subseteq \mathbb{R}$ we obtain immediately
\begin{align*}
0  \leq  \inf_{\substack{C- \text{\textit{closed}},\\ C\subseteq E}} \mu(E \setminus C) \leq \mu(E\setminus E) = 0.
\end{align*}
Hence, $\mu$ is strongly regular by Proposition~\ref{inner=outer}.
On the other hand, the only sets having finite measure are finite sets and so \eqref{OSF-main} can not hold, since $\mathbb{R}$ is uncountable.
\end{proof}

In the following example, we see that the implication \textit{(1)} $\Rightarrow$ \textit{(3)} in Theorem~\ref{main} can fail if we replace \eqref{OSF-main} with the weaker assumption of $\sigma$-finiteness.

\begin{ex}
\label{ex:wkL-1}
There exists a topological measure space $(X,\tau,\mathfrak{M},\mu)$ such that $\mu$ is $\sigma$-finite and Borel regular,  every open set is  $F_\sigma$, $(X,\tau)$ fails 
\eqref{OSF-main}, and \textit{(3)} in Theorem~\ref{main} (the weak Lusin theorem) fails to hold. Consequently, $\mu$ is not strongly regular or outer regular.
\end{ex}
\begin{proof}
We consider the interval $X:=[0,1]$ equipped with Euclidean distance and metric topology. Then, every open set is $F_{\sigma}$. 
Going further, we define a measure $\mu$ on the $\sigma$-algebra of Lebesgue measurable subsets of $[0,1]$ by setting
\[
\mu(A) := \lambda (A) + \sum_{q\in \mathbb{Q}\cap [0,1]} \delta_{q}(A),
\]
where $\lambda$ is the Lebesgue measure and $\delta_{q}$ is the Dirac measure concentrated at $q$. Then \eqref{OSF-main} fails, since every nonempty open set has infinite measure. It is clear that $\mu$ is a $\sigma$-finite, Borel regular measure.

To show that \textit{(3)} in Theorem~\ref{main} fails to hold, let $u:[0,1] \to \mathbb{R}$ be given by  $u:= \chi_{[0,1]\setminus \mathbb{Q}}$. Fix $\varepsilon\in(0,1)$ and let $E\subseteq [0,1]$ be any measurable set, such that $\mu([0,1]\setminus E) < \varepsilon$. Since $\varepsilon<1$, we have
\begin{align*}
[0,1]\setminus E \subseteq [0,1] \setminus \mathbb{Q}
\end{align*}
which means that $\mathbb{Q} \cap [0,1] \subseteq E$. Moreover, $E$ contains at least $1$ irrational number, because if we had $E=\mathbb{Q} \cap [0,1]$, then  $1=\mu([0,1]\setminus E)<1$, which is not possible. As such, from density of rational numbers in $[0,1]$ we conclude that  $u\vert_{E}$ cannot be continuous.
\end{proof}

As the following example illustrates, the implication \textit{(1)} $\Rightarrow$ \textit{(3)} in Theorem~\ref{main} can fail to hold if we omit the assumption that every open set is the union of an $F_\sigma$ set and a null set.
\begin{ex} \label{weakl}
There exists a topological measure space $(X,\tau,\mathfrak{M},\mu)$ such that  $\mu$ is Borel regular,
$(X,\tau)$ satisfies \eqref{OSF-main}, 
not every open set is the union of an $F_\sigma$ set and a null set, and \textit{(3)} in Theorem~\ref{main} (the weak Lusin theorem) fails to hold. Consequently, $\mu$ is not strongly regular.
\end{ex}
\begin{proof}
Let $X:=\{a,b,c\}$. Consider the topology $\tau:=\big\{\varnothing,X,\{c\},\{a,c\},\{b,c\}\big\}$, and define a measure $\mu:\mathfrak{M}\to[0,\infty]$, by setting $\mu :=\delta_a + \delta_b + \delta_c$, where $\mathfrak{M}:=2^X$. In this case, $\mu(X)=3<\infty$ and so $(X,\tau,\mathfrak{M},\mu)$ satisfies \eqref{OSF-main}. 
Moreover, $\mu$ is trivially Borel regular since $Borel_{\tau}(X)=2^X$. We want to show that \textit{(3)} in Theorem~\ref{main} fails. Let $u: X\to\mathbb{R}$ by any nonconstant measurable function, and let $\varepsilon\in(0,1)$. Then, from the definition of $\mu$, we have that $C_\varepsilon=X$ is the only measurable set satisfying $\mu(X\setminus C_\varepsilon)<\varepsilon$. However, if $u\vert_{C_{\varepsilon}}=u$ is continuous then $u$ must be a constant function, which is not possible. Hence, \textit{(3)} in Theorem~\ref{main}  fails to hold, as wanted. 
\end{proof}


\begin{thebibliography}{99}

\bibitem{AGS} {\sc R. Alvarado, P. G\'{o}rka,  A. S\l{}abuszewski},
Compact embeddings of Sobolev, Besov, and Triebel--Lizorkin spaces, arXiv: 2406.18527.
 
 
\bibitem{AM15}  {\sc R. Alvarado, M. Mitrea}, 
Hardy Spaces on Ahlfors-Regular Quasi Metric Spaces, Lecture Notes in Mathematics 2142, Springer, New York, 2015.

\bibitem{AMS} {\sc R. Alvarado, M. Mitrea, and B. Schmutzler}, {Characterizing the density of continuous functions in Lebesgue spaces via the regularity of the measure}, (preprint)

\bibitem{Bog} {\sc V. I. Bogachev}, Measure theory. Vol. I, II.
Springer-Verlag, Berlin, 2007. Vol. I: xviii+500 pp., Vol. II: xiv+575 pp.



\bibitem{C27} {\sc L. W. Cohen}, A new proof of Lusin's theorem, Fund. Math. 9 (1927), 122-123.

\bibitem{EG15} {\sc L. C. Evans, R. F. Gariepy},
Measure theory and fine properties of functions.
Revised edition.
Textb. Math.
CRC Press, Boca Raton, FL, 2015. xiv+299 pp.


\bibitem{Federer} {\sc H. Federer}, Geometric measure theory.
Die Grundlehren der mathematischen Wissenschaften, Band 153
Springer-Verlag New York, Inc., New York, 1969. xiv+676 pp.

\bibitem{Feldman} {\sc M. B. Feldman}, A proof of Lusin's theorem, 
Amer. Math. Monthly 88 (1981), no.3, 191–192.

\bibitem{HKST} {\sc J. Heinonen, P. Koskela, N. Shanmugalingam, J. T. Tyson}, Sobolev spaces on metric measure spaces,
New Math. Monogr., 27 Cambridge University Press, Cambridge, 2015, xii+434 pp.


\bibitem{L12} {\sc N. Lusin}, Sur les propri\'et\'es des fonctions mesurables, Comptes rendus de l'Acad\'emie des Sciences de Paris 154 (1912), 1688--1690.

\bibitem{MMM23} {\sc D. Mitrea, I. Mitrea and M. Mitrea,} Geometric Harmonic Analysis I. A Sharp Divergence Theorem with Nontangential Pointwise Traces. Developments in Mathematics, 72. Springer, Cham, 2022. xxviii+924 pp.

\bibitem{Niizeki} {\sc S. Niizeki}, 
Remarks on the Lusin's theorem for the Lebesgue measurable function on $\mathbf{R}^N$,
Mem. Fac. Sci. Kochi Univ. Ser. A Math. 20 (1999), 127–134.


\bibitem{Stone} {\sc A. H. Stone}, Lusin's theorem, (English summary)
Atti Sem. Mat. Fis. Univ. Modena 44 (1996), no.2, 351–357.

\bibitem{Wage} {\sc M. L. Wage}, A generalization of Lusin's theorem, 
Proc. Amer. Math. Soc. 52 (1975), 327–332.


\bibitem{Willard} {\sc S. Willard}, General Topology. 
Reading, Mass. :Addison-Wesley Pub. Co., 1970. 


\bibitem{Zakon} {\sc E. Zakon}, A remark on the theorems of Lusin and Egoroff, 
Canad. Math. Bull. 7 (1964), 291–295.

\end{thebibliography}
\end{document}